\documentclass[a4paper,oneside,reqno,12pt]{amsart}

\usepackage{cmap}
\usepackage[T1]{fontenc}
\usepackage[utf8]{inputenc}
\usepackage[english]{babel}
\usepackage{csquotes}
\usepackage{lmodern}

\usepackage{amsmath,amssymb,amsfonts,mathtools}

\usepackage[
    a4paper,
    hmargin=2.2cm,
    vmargin=3cm
]{geometry}

\usepackage{xcolor}
\usepackage{graphicx}
\usepackage{subcaption}
\usepackage{tikz}
\usetikzlibrary{calc}

\usepackage{hyperref}
\hypersetup{
    colorlinks = true,
    urlcolor   = blue,
    linkcolor  = red,
    citecolor  = green
}

\newtheorem{theorem}{Theorem}
\newtheorem{lemma}[theorem]{Lemma}
\newtheorem{corollary}[theorem]{Corollary}
\newtheorem{conjecture}{Conjecture}

\DeclareMathOperator{\vol}{vol}

\newcommand{\E}{\mathbb{E}}

\title{Optimal Partial Plank Coverings}

\author[E. Bakaev]{Egor Bakaev}
\address{Egor Bakaev, 
Department of Computer Science, University of Copenhagen}

\author[A. Polyanskii]{Alexander~Polyanskii}

\address{Alexander Polyanskii,
Department of Mathematics, Emory University,
400 Dowman Dr, Atlanta, GA 30322, USA}
\email{\href{mailto:apolian@emory.edu}{apolian@emory.edu}}
\urladdr{\url{https://polyanskii.com}}


\begin{document}

\begin{abstract}
A plank of width $w$ in a Euclidean space is the set of points lying between two parallel hyperplanes at distance $w$ from each other. Bang's theorem says that if a family of planks covers a convex body $K$, then their total width is at least the width of $K$, that is, the width of the thinnest plank containing $K$.

We study a quantitative variant of this problem in the case where the total width of the planks is fixed. How should the planks be placed so as to cover as much of the volume of the body as possible?

For the central case where $K$ is a Euclidean ball, K\'aroly Bezdek asked whether the optimal arrangement consists of a single plank centered at the origin. We give an affirmative answer to this question. We also show that for every planar convex body an optimal partial covering is attained by a single plank.
\end{abstract}

\maketitle

\section{Introduction}

Let $\E^d$ denote the $d$-dimensional Euclidean space, with Euclidean norm
$|\cdot|$. A \emph{convex body} in $\E^d$ is a compact convex set with
nonempty interior. A \emph{plank} of width $w$ is the closed (or open) set of
points lying between two parallel hyperplanes at distance $w$ from each
other. For a convex body $K$, its \emph{width}, denoted by $w(K)$, is
the width of a thinnest plank containing~$K$.

The classical plank problem, attributed to Tarski~\cite{Tarski1932}, asks
whether a convex body $K$ can be covered by closed planks whose total width is
smaller than $w(K)$. Bang proved that this is impossible
\cite{Bang1950,Bang1951}. This result, now known as Bang's plank theorem,
initiated a broad line of research in discrete and convex geometry, as well as in
functional analysis.

Revisiting these questions, K\'aroly Bezdek~\cite[Section~5]{Bezdek2013} formulated a quantitative partial-covering
problem for the Euclidean ball in which the widths of the planks are
prescribed. Let
\(B^d\)
be the unit ball centered at the origin, and let $w_1,\ldots,w_n>0$ satisfy
\[
    W:=\sum_{i=1}^n w_i<2=w(B^d).
\]
Bezdek~\cite[Problem~5.1]{Bezdek2013} asked whether, among all planks
$P_1,\ldots,P_n$ with $w(P_i)=w_i$, the covered volume
\[
    \vol_d\!\left(B^d\cap\bigcup_{i=1}^n P_i\right)
\]
is maximized if and only if the planks are parallel, in an arbitrary common direction, and their union, up to boundaries, is a single plank of width $W$ centered at the origin.
He proved the assertion in two cases. For two planks in every dimension, he first established the planar case and then integrated over two-dimensional sections. 
For an arbitrary finite family in dimension~$3$, he used the
spherical-surface argument originating in Moese's solution to a different
problem posed by Tarski~\cite{Moese1932}. Tarski subsequently adapted this
argument to the disk case of the plank problem~\cite{Tarski1932}. Bezdek
applied it to concentric spheres and then
integrated over the radius.

We pose the following conjecture.

\begin{conjecture}\label{conj:fixed-width}
Let $K\subset\E^d$ be a convex body and let $W>0$. There exists a plank $P$ of
width~$W$ such that, for every finite family of planks $P_1,\ldots,P_n$
satisfying
\[
    \sum_{i=1}^n w(P_i)=W,
\]
one has
\[
    \vol_d\!\left(K\cap\bigcup_{i=1}^n P_i\right)
    \leq \vol_d(K\cap P).
\]
\end{conjecture}

In words, among all families of planks with prescribed total width, an
optimal family can be chosen to consist of a single plank.
For $K=B^d$ and $0<W<2$, a plank of width $W$ has largest intersection with
the ball when it is centered at the origin. 

We prove the ball case of Conjecture~\ref{conj:fixed-width} in Theorem \ref{thm:main}. The theorem, together with the subsequent uniqueness discussion, solves Bezdek's problem.

He also proved Conjecture~\ref{conj:fixed-width} for triangles \cite[Corollary~5.8]{Bezdek2013}. We prove it for every planar convex body; see Theorem~\ref{thm:planar-bodies}.

\section{A reduction for plank families}

For sets $A,B\subset\E^d$, their \emph{Minkowski sum} is defined by
\[
    A+B:=\{a+b:a\in A,\ b\in B\}.
\]
We use the convention that $A+\varnothing=\varnothing$.
For sets
$A,B\subset\E^d$ we call
\[
    A\div B
    :=
    \bigcap_{b\in B}(A-b)
    =
    \{x\in\mathbb{R}^d:B+x\subseteq A\}
\]
the \emph{Minkowski difference} of $A$ and $B$; the corresponding
operation is called \emph{Minkowski subtraction}
\cite[Section~3.1, pp.~146--147]{schneider2013convex}.
In particular, $A\div\varnothing=\E^d$. 
This should not be confused with the set
$A-B=A+(-B).$

The following inclusion is the form of concavity under
Minkowski subtraction. (The two-set case of the lemma follows from
\cite[Lemma~3.1.13]{schneider2013convex}; the finite-set case follows by
induction.)

\begin{lemma}
\label{lem:minkowski-subtraction-concavity}
Let $K\subset\E^d$ be convex, let
$A_1,\ldots,A_n\subset\E^d$ be arbitrary sets, and let
$\lambda_1,\ldots,\lambda_n \geq 0$ satisfy $\sum_{i=1}^n\lambda_i=1$. Then
\[
    \lambda_1(K\div A_1)+\cdots+\lambda_n(K\div A_n)
    \subseteq
    K\div\bigl(\lambda_1A_1+\cdots+\lambda_nA_n\bigr).
\]
\end{lemma}

\begin{proof}
If $A_i=\varnothing$ for some $i$, then the right-hand side is
$K\div\varnothing=\E^d$.
If $K \div A_i=\varnothing$ for some $i$, then the left-hand side is
$\varnothing$. In both these cases the inclusion is immediate. We may thus
assume that all the sets $A_i$ and $K \div A_i$ are nonempty.
Take $x_i\in K\div A_i$. For arbitrary $a_i\in A_i$, one has
$x_i+a_i\in K$. Hence convexity of $K$ gives
\[
    \sum_{i=1}^n\lambda_i x_i+\sum_{i=1}^n\lambda_i a_i
    =
    \sum_{i=1}^n\lambda_i(x_i+a_i)
    \in K.
\]
Since the $a_i$ were arbitrary, the asserted inclusion follows.
\end{proof}

Given open planks $P_1,\ldots,P_n$, write each of them as
\[
    P_i=
    \left\{
        x\in\E^d:
        |\langle x,a_i\rangle-\alpha_i|<|a_i|^2
    \right\},
\]
where $a_i\neq0$. Thus the width of $P_i$ is $2|a_i|$. We associate with this family the \emph{Bang set}, which may be viewed as a
``discrete zonotope'':
\[
    Z_{\mathcal P}:=\sum_{i=1}^n\{-a_i,a_i\}.
\]

\subsection{The Bang map}

The construction in this subsection is essentially a reformulation of
Bang's classical argument in the proof of the plank theorem
\cite{Bang1950,Bang1951}. We present it as a piecewise-translation map,
since the volume-preserving property of this map will be important below.

For $\varepsilon=(\varepsilon_1,\ldots,\varepsilon_n)\in\{\pm1\}^n$, put
\[
    v(\varepsilon):=\sum_{i=1}^n\varepsilon_i a_i,
    \qquad
    \alpha(\varepsilon):=\sum_{i=1}^n\varepsilon_i\alpha_i.
\]
For $t\in\E^d$, define the function
$
\Phi_t\colon\{\pm1\}^n\to\mathbb R
$
by
\[
\Phi_t(\varepsilon)
:=|t+v(\varepsilon)|^2-2\alpha(\varepsilon).
\]

Fix an ordering of all the sign vectors in $\{\pm1\}^n$. For each
$t\in\E^d$, let $\varepsilon(t)$ be the first sign vector in this
ordering at which $\Phi_t$ attains its maximum. 
Define the {\em Bang map}
$
    F\colon\E^d\to\E^d
$
by
\[
    F(t):=t+v(\varepsilon(t)).
\]

For centered planks, the construction below is the farthest-point decomposition of the Bang set used in \cite[Section~4]{balitskiy2021multi} and the maximal-norm construction in \cite[Lemma~6]{polyanskii2021cap}.

For
\(\varepsilon\in\{\pm 1\}^n\), define
\[
    C_\varepsilon
    :=
    \left\{
        t\in\mathbb{E}^d :
        \varepsilon(t)=\varepsilon
    \right\}.
\]
The sets $C_\varepsilon$  form a partition of $\mathbb{E}^d$.
Also for
\(\varepsilon\in\{\pm 1\}^n\), define
\[
    Q_\varepsilon
    :=
    \left\{
        x\in\mathbb{E}^d :
        \varepsilon_i
        \bigl(
            \langle x,a_i\rangle-\alpha_i
        \bigr)
        \geq
        |a_i|^2
        \text{ for all }i
    \right\}.
\]

\begin{lemma}\label{lem:Q-disjoint}
The sets \(Q_\varepsilon\) are pairwise disjoint, and each
\(Q_\varepsilon\) is disjoint from all the open planks
\(P_1,\ldots,P_n\).
\end{lemma}
\begin{proof}
Let \(\varepsilon, \delta \in\{\pm 1\}^n\) and 
\(\varepsilon\neq\delta\). Choose \(j\) such that
\(\delta_j=-\varepsilon_j\). If
\(x\in Q_\varepsilon\cap Q_\delta\), then
\[
    \varepsilon_j
    \bigl(
        \langle x,a_j\rangle-\alpha_j
    \bigr)
    \geq
    |a_j|^2
\quad 
\text{and}
\quad
    -\varepsilon_j
    \bigl(
        \langle x,a_j\rangle-\alpha_j
    \bigr)
    \geq
    |a_j|^2,
\]
which is impossible. Thus, the sets \(Q_\varepsilon\) are pairwise
disjoint.

Now, if \(x\in Q_\varepsilon\), then for every \(i\),
\[
    \left|
        \langle x,a_i\rangle-\alpha_i
    \right|
    \geq
    |a_i|^2.
\]
Hence \(x\notin P_i\) for every \(i\), so \(Q_\varepsilon\) is
disjoint from all the open planks.
\end{proof}

The following lemma shows, in particular, that the map
$F\colon\mathbb{E}^d\to\mathbb{E}^d$ is injective.

\begin{lemma}
\label{lem:bang-cell-inclusion}
 For every \(\varepsilon\in\{\pm 1\}^n\), the set
\(C_\varepsilon\) is measurable and
\[
    F(C_\varepsilon)
    =
    C_\varepsilon+v(\varepsilon)
    \subseteq
    Q_\varepsilon.
\]
\end{lemma}

\begin{proof}
Recall that $\varepsilon(t)$ is chosen as the first maximizer of
$\Phi_t$ with respect to the fixed ordering of $\{\pm 1\}^n$. Thus,
$t\in C_\varepsilon$ precisely when
$\Phi_t(\varepsilon)>\Phi_t(\delta)$
for every $\delta$ preceding $\varepsilon$, and
$\Phi_t(\varepsilon)\geq\Phi_t(\delta)$
for every $\delta$ following $\varepsilon$. Since each difference
$\Phi_t(\varepsilon)-\Phi_t(\delta)$ is affine in $t$, the set
$C_\varepsilon$ is a finite intersection of open and closed half-spaces,
and therefore measurable.

Take \(t\in C_\varepsilon\) and put
\[
    x
    :=
    F(t)
    =
    t+v(\varepsilon).
\]
For \(j\in\{1,\ldots,n\}\), let \(\varepsilon^{(j)}\) be obtained
from \(\varepsilon\) by changing its \(j\)-th coordinate. By the
maximality of \(\varepsilon\),
\[
\begin{aligned}
    0
    &\leq
    \Phi_t(\varepsilon)
    -
    \Phi_t\bigl(\varepsilon^{(j)}\bigr)
    =
    4\left(
        \varepsilon_j
        \bigl(
            \langle x,a_j\rangle-\alpha_j
        \bigr)
        -
        |a_j|^2
    \right).
\end{aligned}
\]
Consequently,
\[
    \varepsilon_j
    \bigl(
        \langle x,a_j\rangle-\alpha_j
    \bigr)
    \geq
    |a_j|^2
\]
for every \(j\), and therefore \(x\in Q_\varepsilon\). This proves
$
    F(C_\varepsilon)
    \subseteq
    Q_\varepsilon.
$
\end{proof}

\begin{corollary}
\label{lem:bang-map-volume}
For every measurable set \(T\subset\mathbb{E}^d\), the set \(F(T)\)
is measurable,
\[
    \vol_d(F(T))
    =
    \vol_d(T),
\]
and \(F(T)\) is disjoint from all the open planks
\(P_1,\ldots,P_n\).
\end{corollary}

\begin{proof}
Since the sets \(C_\varepsilon\) form a measurable partition of
\(\mathbb{E}^d\),
\[
    F(T)
    =
    \bigcup_{\varepsilon\in\{\pm 1\}^n}
    \big(
        (T\cap C_\varepsilon)+v(\varepsilon)
    \big).
\]
By Lemma~\ref{lem:bang-cell-inclusion},
$
    (T\cap C_\varepsilon)+v(\varepsilon)
    \subseteq
    Q_\varepsilon.
$
By Lemma~\ref{lem:Q-disjoint}, the sets
\(Q_\varepsilon\) are pairwise disjoint. Hence the sets in the union
above are also pairwise disjoint. Therefore,
\[
\begin{aligned}
    \vol_d(F(T))
    &=
    \sum_{\varepsilon\in\{\pm 1\}^n}
    \vol_d
    \big(
        (T\cap C_\varepsilon)+v(\varepsilon)
    \big)
    \\
    &=
    \sum_{\varepsilon\in\{\pm 1\}^n}
    \vol_d(T\cap C_\varepsilon)
    \\
    &=
    \vol_d(T).
\end{aligned}
\]
The same inclusion and the fact that each $Q_\varepsilon$ is disjoint from all the open planks show that $F(T)$ is also disjoint from all the open planks.
\end{proof}

\begin{lemma}\label{lem:bang-inner}
Let $K\subset\E^d$ be a convex body. Then
\[
    \vol_d\left(K\setminus\bigcup_{i=1}^nP_i\right)
    \geq
    \vol_d(K\div Z_{\mathcal P}).
\]
\end{lemma}

\begin{proof}
Apply Corollary~\ref{lem:bang-map-volume} to $T:=K\div Z_{\mathcal P}$.
For every $\varepsilon\in\{\pm1\}^n$, we have $v(\varepsilon)\in Z_{\mathcal P}$, and hence
$t+v(\varepsilon)\in K$ for every $t\in T$. Thus $F(T)$ is contained in
$K\setminus\bigcup_{i=1}^n P_i$, while $F$ preserves volume.
\end{proof}

\subsection{Reduction to a single direction}

\begin{lemma}\label{lem:discrete-reduction}
Let $K\subset\E^d$ be a convex body, let $a_1,\ldots,a_n\in\E^d$, and put
\[
    Z:=\sum_{i=1}^n\{-a_i,a_i\},
    \qquad
    r:=\sum_{i=1}^n|a_i|.
\]
If $r>0$, then
\[
    \vol_d(K\div Z)
    \geq
    \min_{u\in S^{d-1}}
    \vol_d\bigl(K\div\{-ru,ru\}\bigr).
\]
\end{lemma}

\begin{proof}
Without loss of generality, none of the vectors $a_1,\ldots,a_n$ is zero. Write
$
    u_i:=\frac{a_i}{|a_i|},
$
$
    \lambda_i:=\frac{|a_i|}{r}.
$
Then $\lambda_i>0$, $\sum_{i=1}^n\lambda_i=1$, and
$
    Z=\sum_{i=1}^n\lambda_i \{-ru_i,ru_i\}.
$

If $K\div \{-ru_i,ru_i\}$ is empty for some $i$, then the right-hand side of the
claimed inequality is zero, so the result is immediate. Otherwise,
Lemma~\ref{lem:minkowski-subtraction-concavity} gives
\[
    \sum_{i=1}^n \lambda_i\bigl(K\div \{-ru_i,ru_i\}\bigr)
    \subseteq K\div Z.
\]
The Brunn--Minkowski inequality now yields
\begin{align*}
    \vol_d(K\div Z)^{1/d}
    &\geq
    \sum_{i=1}^n \lambda_i
    \vol_d\bigl(K\div\{-ru_i,ru_i\}\bigr)^{1/d} \\
    &\geq
    \min_{u\in S^{d-1}}
    \vol_d\bigl(K\div\{-ru,ru\}\bigr)^{1/d}.
\end{align*}

Raising to the $d$-th power proves the result.
\end{proof}

Combining the preceding two lemmas gives the common estimate used in both
special cases below.

\begin{corollary}\label{cor:uncovered-reduction}
Let $K\subset\E^d$ be a convex body, and let $P_1,\ldots,P_n$ be planks whose
total width is at most $W$. Put $R:=W/2$. Then
\[
    \vol_d\left(K\setminus\bigcup_{i=1}^nP_i\right)
    \geq
    \min_{u\in S^{d-1}}
    \vol_d\bigl(K\div\{-Ru,Ru\}\bigr).
\]
\end{corollary}

\begin{proof}
After changing each plank on its boundary, write it in the normalized form
above and put $r:=\sum_{i=1}^n|a_i|$. Then $r\leq R$. If $r=0$, the assertion
is immediate. Otherwise, Lemmas~\ref{lem:bang-inner} and
\ref{lem:discrete-reduction} give
\[
    \vol_d\left(K\setminus\bigcup_{i=1}^n P_i\right)
    \geq
    \min_{u\in S^{d-1}}
    \vol_d\bigl(K\div\{-ru,ru\}\bigr).
\]
For fixed $u$, the sets $K\div\{-tu,tu\}$ decrease with respect to
inclusion as $t$ increases: if $0\leq s\leq t$ and $x\pm tu\in K$, then
convexity gives $x\pm su\in K$. Since $r\leq R$, the right-hand side is
at least the claimed minimum.
\end{proof}

The following lemma is essentially another key step in Bang's proof.

\begin{lemma}
\label{lem:translate-inclusion}
Let $K$ be a convex body with width $w = w(K)$, let \(0<r<w\) and let $u$ be any unit vector. Then some translate of
$
    \left(1-\frac{r}{w}\right)K
$
is contained in
$
    K\div\left\{-\frac{r}{2}u,\frac{r}{2}u\right\}.
$
\end{lemma}

\begin{proof}
First, there exist $x,y \in K$ such that $x-y = wu$. Indeed, otherwise the convex bodies $K-wu/2$ and $K+wu/2$ would be disjoint and could be strictly separated by a hyperplane. If $n$ is a unit normal
to such a hyperplane, then the width of $K$ in the direction $n$ would be less than
$
    w |\langle u,n\rangle|
    \leq w,
$
contrary to the definition of $w$ as the minimal width.

Put
$\rho:=\frac{r}{w},$
$c:=\frac{\rho}{2}(x+y).$
We show that \[
    c+(1-\rho)K
    \subseteq
    K\div\left\{-\frac{r}{2}u,\frac{r}{2}u\right\}.
\]

For every \(z\in K\), using
$
    \frac{r}{2}u=\frac{\rho}{2}(x-y),
$
we obtain
\[
    c+(1-\rho)z+\frac{r}{2}u
    =(1-\rho)z+\rho x\in K
\]
and
\[
    c+(1-\rho)z-\frac{r}{2}u
    =(1-\rho)z+\rho y\in K,
\]
since both points are convex combinations of points of \(K\). This proves the assertion.
\end{proof}

\section{Euclidean ball}

The reduction above is sharp for a ball because all directions are
equivalent and $B^d\div\{-ru,ru\}$ has the same volume as the part left
uncovered by the corresponding central plank.

\begin{lemma}\label{lem:ball-lens-plank}
Let $u\in S^{d-1}$ and $0<r<1$. Put
\[
    L_{r,u}:=B^d\div\{-ru,ru\}
\]
and let
\[
    P_{r,u}
    :=
    \{x\in\E^d:|\langle x,u\rangle|\leq r\}
\]
be the central plank of width $2r$ orthogonal to $u$. Then
\[
    \vol_d(L_{r,u})
    =
    \vol_d(B^d\setminus P_{r,u}).
\]
\end{lemma}

\begin{proof}
Let
$
    u^\perp:=\{x\in\E^d:\langle x,u\rangle=0\}.
$
We split both $L_{r,u}$ and $B^d\setminus P_{r,u}$ into the two parts
lying on opposite sides of $u^\perp$, and show that the corresponding parts are translates of each other; see Figure \ref{fig:ball-lens}.

We claim that
$
    L^++ru
    =
    B^+,
$
where
\[
    L^+
    :=
    L_{r,u}\cap\{x:\langle x,u\rangle>0\},
\qquad
    B^+
    :=
    B^d\cap\{x:\langle x,u\rangle>r\}.
\]

Indeed, for $x\in\E^d$,
\[
\begin{aligned}
    x\in L^++ru
    &\Longleftrightarrow
    x-ru\in L_{r,u}
    \ \text{and}\
    \langle x-ru,u\rangle>0
    \\
    &\Longleftrightarrow
    x\in B^d,\quad
    x-2ru\in B^d,\quad
    \langle x,u\rangle>r
    \\
    &\Longleftrightarrow
    x\in B^d
    \ \text{and}\
    \langle x,u\rangle>r.
\end{aligned}
\]
For the last equivalence, if $x\in B^d$ and $\langle x,u\rangle>r$, then
\[
    |x-2ru|^2
    =
    |x|^2-4r\langle x,u\rangle+4r^2
    <
    |x|^2
    \leq 1.
\]

Therefore,
$
    L^++ru
    =
    B^+.
$
Similarly,
$
    L^--ru
    =
    B^-,
$
where
\[
    L^-
    :=
    L_{r,u}\cap\{x:\langle x,u\rangle<0\},
\qquad
    B^-
    :=
    B^d\cap\{x:\langle x,u\rangle<-r\}.
\]

Since 
$
    L_{r,u}\setminus\bigl(L^+\cup L^-\bigr)
    \subset u^\perp
$
has measure zero, and
$B^d\setminus P_{r,u}$ is a disjoint union of $B^+$ and $B^-$,
the result follows.
\end{proof}

\begin{figure}
    \centering
    \begin{tikzpicture}[
    line cap=round,
    line join=round,
    every node/.style={font=\small}
]
    \def\ballradius{2.00}
    \def\lensshift{1.10}
    \pgfmathsetmacro{\plankhalfheight}{sqrt(\ballradius*\ballradius-\lensshift*\lensshift)}

    % Central plank and the two caps, clipped to the ball.
    \begin{scope}
        \clip (0,0) circle (\ballradius);
        \fill[black!7]  (-\lensshift,-\ballradius)
            rectangle (\lensshift,\ballradius);
        \fill[black!20] (-\ballradius,-\ballradius)
            rectangle (-\lensshift,\ballradius);
        \fill[black!60] (\lensshift,-\ballradius)
            rectangle (\ballradius,\ballradius);
    \end{scope}

    % Lens L_{r,u}, split by u^\perp.
    \begin{scope}
        \clip (-\lensshift,0) circle (\ballradius);
        \clip ( \lensshift,0) circle (\ballradius);
        \fill[black!20] (-\ballradius-\lensshift,-\ballradius)
            rectangle (0,\ballradius);
        \fill[black!60] (0,-\ballradius)
            rectangle (\ballradius+\lensshift,\ballradius);
    \end{scope}

    % Boundaries of the two translated balls defining the lens.
    \draw[black!35,thick,dash pattern=on 5pt off 4pt]
        (-\lensshift,0) circle (\ballradius);
    \draw[black!35,thick,dash pattern=on 5pt off 4pt]
        ( \lensshift,0) circle (\ballradius);

    % Original ball and plank boundaries.
    \draw[black,very thick] (0,0) circle (\ballradius);
    \draw[black!65,thick]
        (-\lensshift,-\plankhalfheight) -- (-\lensshift,\plankhalfheight);
    \draw[black!65,thick]
        ( \lensshift,-\plankhalfheight) -- ( \lensshift,\plankhalfheight);

    % Separating hyperplane and marked points.
    \draw[black!65,thick]
        (0,-\ballradius-0.18) -- (0,\ballradius+0.18)
        node[above,black] {$u^\perp$};

    \fill (0,0) circle (1.5pt);
    \node[anchor=east,xshift=-2pt,yshift=-4pt] at (0,0) {$0$};

    \fill (-\lensshift,0) circle (1.5pt);
    \fill ( \lensshift,0) circle (1.5pt);
    \node[anchor=east,xshift=+1pt,yshift=-4pt]
        at (-\lensshift,0) {$-ru$};
    \node[anchor=west,xshift=3pt,yshift=-4pt, text=white] at (\lensshift,0) {$ru$};

    % Region labels.
    \node at (-0.47,0.86) {$L^-$};
    \node[white] at (0.4,0.86) {$L^+$};
    \node at (-1.50,0.86) {$B^-$};
    \node[white] at (1.50,0.86) {$B^+$};
    \node at (0.35,-\ballradius-0.35) {$B^d$};
\end{tikzpicture}
    \caption{The two parts
    \(L^-\) and \(L^+\) are translates of 
    \(B^-\) and \(B^+\).}
    \label{fig:ball-lens}
\end{figure}
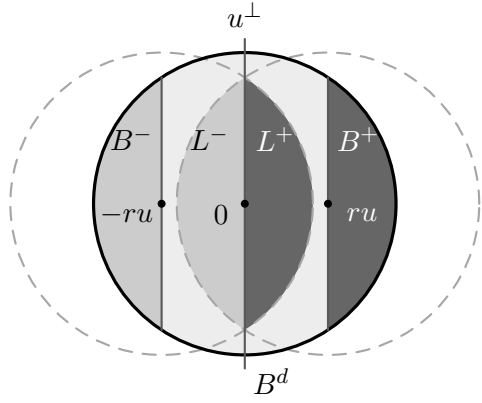

\begin{theorem}\label{thm:main}
For the unit ball, among all finite collections of planks whose total width
$W<2$, the largest covered volume is attained by a single plank centered at
the origin of width~$W$.
\end{theorem}

\begin{proof}
Let the total width of the given planks be $W<2$, and put $r:=W/2$. By
Corollary~\ref{cor:uncovered-reduction},
\[
    \vol_d\left(B^d\setminus\bigcup_{i=1}^n P_i\right)
    \geq
    \min_{u\in S^{d-1}}\vol_d(B^d\div\{-ru,ru\}).
\]
Rotational symmetry makes the quantity on the right independent of $u$.
By Lemma~\ref{lem:ball-lens-plank}, for every $u\in S^{d-1}$ it equals
$
    \vol_d(B^d\setminus P_{r,u}),
$
where $P_{r,u}$ is the central plank of width $2r=W$. Taking complements in
$B^d$ gives
\[
    \vol_d\left(B^d\cap\bigcup_{i=1}^n P_i\right)
    \leq
    \vol_d(B^d\cap P_{r,u}),
\]
which proves the theorem.
\end{proof}

We also show that, for $d\geq 2$, every optimal arrangement of planks for the ball $B^d$ has, up to boundaries, the same union as a single central plank. If equality holds in the final estimate, then
equality must also hold in the Brunn--Minkowski step, which implies that all the summands are homothetic. The bodies
$B^d\div\{-ru_i,ru_i\}$
have equal volume and are symmetric about the origin, so if they are homothetic then they are in fact equal. It follows that all the planks are
parallel. In the parallel case, the strict monotonicity of the volumes of the
sections of the ball implies that equality can occur only when the union of
the planks is the central plank of width~$W$. This completes the solution of Bezdek's problem~\cite[Problem~5.1]{Bezdek2013}.

\section{Planar convex bodies}

For planar convex bodies, the same reduction is sharp for a different
reason: the set obtained by subtracting a pair of points can be cut and translated
to form the part lying outside a suitable plank.

\begin{lemma}\label{lem:planar-cap-lens}
Let $K\subset\E^2$ be a convex body, and let $z\in\E^2$ satisfy
$0<|z|<w(K)$.  Then there is a plank $P_z$ of width at most $|z|$ such
that
\[
    \vol_2(K\setminus P_z)
    =
    \vol_2\bigl(K\div\{-z/2,z/2\}\bigr).
\]
\end{lemma}

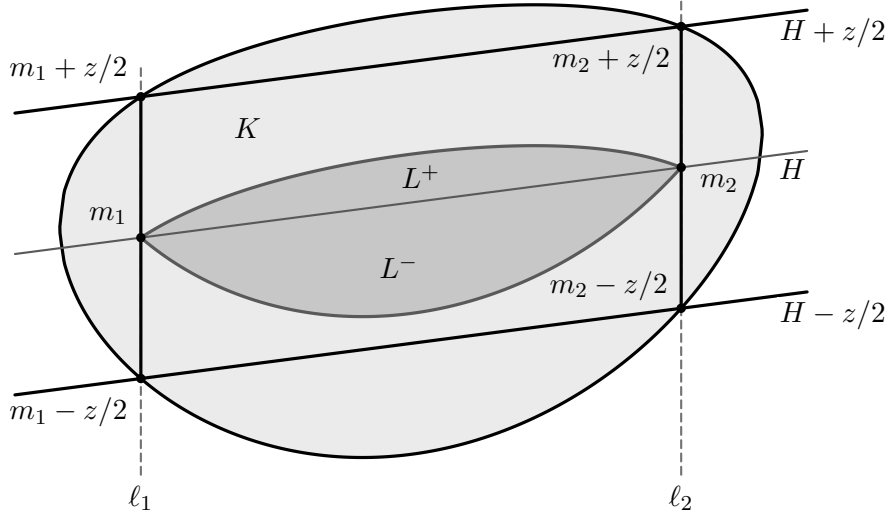
\begin{figure}
    \centering
    \begingroup
\begin{tikzpicture}[
  x=1.08cm,y=1.08cm,
  line cap=round,line join=round,
  every node/.style={font=\small,text=black},
  endpoint/.style={circle,fill=black,inner sep=1.25pt},
  chord/.style={black,very thick},
  guide/.style={black!55,densely dashed,thick}
]

\def\vclA{4.3}
\def\vclD{5.4}
\def\vclEll{3.45}
\pgfmathsetmacro{\vclXs}{\vclA*sqrt(1-(\vclEll/\vclD)^2)}
\pgfmathsetmacro{\vclCL}{0.13*(-\vclXs)+0.035*\vclXs*\vclXs-0.20}
\pgfmathsetmacro{\vclCR}{0.13*( \vclXs)+0.035*\vclXs*\vclXs-0.20}
\pgfmathsetmacro{\vclYLb}{\vclCL-\vclEll/2}
\pgfmathsetmacro{\vclYLt}{\vclCL+\vclEll/2}
\pgfmathsetmacro{\vclYRb}{\vclCR-\vclEll/2}
\pgfmathsetmacro{\vclYRt}{\vclCR+\vclEll/2}
\pgfmathsetmacro{\vclSlope}{(\vclYRb-\vclYLb)/(2*\vclXs)}
\def\vclXmin{-4.85}
\def\vclXmax{4.85}
\def\vclYmin{-3.15}

% Convex body K.
\path[fill=black!8]
  plot[domain=-\vclA:\vclA,samples=180,smooth]
    (\x,{0.13*\x+0.035*\x*\x-0.20
          +2.7*sqrt(max(0,1-(\x/\vclA)^2))})
  --
  plot[domain=\vclA:-\vclA,samples=180,smooth]
    (\x,{0.13*\x+0.035*\x*\x-0.20
          -2.7*sqrt(max(0,1-(\x/\vclA)^2))})
  -- cycle;

% The two parts L^+ and L^- (separated by H).
\path[fill=black!22]
  plot[domain=-\vclXs:\vclXs,samples=140,smooth]
    (\x,{0.13*\x+0.035*\x*\x-0.20
          +2.7*sqrt(max(0,1-(\x/\vclA)^2))-\vclEll/2})
  --
  plot[domain=\vclXs:-\vclXs,samples=140,smooth]
    (\x,{0.13*\x+0.035*\x*\x-0.20
          -2.7*sqrt(max(0,1-(\x/\vclA)^2))+\vclEll/2})
  -- cycle;

% Boundary of K and the two outer boundary arcs of L.
\draw[black,very thick]
  plot[domain=-\vclA:\vclA,samples=180,smooth]
    (\x,{0.13*\x+0.035*\x*\x-0.20
          +2.7*sqrt(max(0,1-(\x/\vclA)^2))})
  --
  plot[domain=\vclA:-\vclA,samples=180,smooth]
    (\x,{0.13*\x+0.035*\x*\x-0.20
          -2.7*sqrt(max(0,1-(\x/\vclA)^2))})
  -- cycle;

\draw[black!65,very thick]
  plot[domain=-\vclXs:\vclXs,samples=140,smooth]
    (\x,{0.13*\x+0.035*\x*\x-0.20
          +2.7*sqrt(max(0,1-(\x/\vclA)^2))-\vclEll/2});
\draw[black!65,very thick]
  plot[domain=-\vclXs:\vclXs,samples=140,smooth]
    (\x,{0.13*\x+0.035*\x*\x-0.20
          -2.7*sqrt(max(0,1-(\x/\vclA)^2))+\vclEll/2});

% The two vertical lines ell_1 and ell_2.
\pgfmathsetmacro{\vclYTopLeft}{\vclYLt+0.35}
\pgfmathsetmacro{\vclYTopRight}{\vclYRt+0.35}
\draw[guide] (-\vclXs,\vclYmin) -- (-\vclXs,\vclYTopLeft);
\draw[guide] ( \vclXs,\vclYmin) -- ( \vclXs,\vclYTopRight);
\node[below] at (-\vclXs,\vclYmin) {$\ell_1$};
\node[below] at ( \vclXs,\vclYmin) {$\ell_2$};

% The two vertical sections of length |z|.
\coordinate (vclMOne) at (-\vclXs,\vclCL);
\coordinate (vclMTwo) at ( \vclXs,\vclCR);
\draw[chord] (-\vclXs,\vclYLb) -- (-\vclXs,\vclYLt);
\draw[chord] ( \vclXs,\vclYRb) -- ( \vclXs,\vclYRt);

% The lines H-z/2, H, and H+z/2.
\draw[black,very thick]
  (\vclXmin,{\vclYLb+\vclSlope*(\vclXmin+\vclXs)}) --
  (\vclXmax,{\vclYLb+\vclSlope*(\vclXmax+\vclXs)})
  node[pos=0.96,below right,fill=white,inner sep=1pt] {$H-z/2$};
\draw[black!65,thick]
  (\vclXmin,{\vclCL+\vclSlope*(\vclXmin+\vclXs)}) --
  (\vclXmax,{\vclCL+\vclSlope*(\vclXmax+\vclXs)})
  node[pos=0.96,below right,fill=white,inner sep=1pt] {$H$};
\draw[black,very thick]
  (\vclXmin,{\vclYLt+\vclSlope*(\vclXmin+\vclXs)}) --
  (\vclXmax,{\vclYLt+\vclSlope*(\vclXmax+\vclXs)})
  node[pos=0.96,below right,fill=white,inner sep=1pt] {$H+z/2$};

% Midpoints and the four endpoints of the extreme sections.
\node[endpoint] at (vclMOne) {};
\node[endpoint] at (vclMTwo) {};
\node[above left=1pt] at (vclMOne) {$m_1$};
\node[right=3pt,yshift=-6pt] at (vclMTwo) {$m_2$};

\node[endpoint] at (-\vclXs,\vclYLb) {};
\node[endpoint] at (-\vclXs,\vclYLt) {};
\node[endpoint] at ( \vclXs,\vclYRb) {};
\node[endpoint] at ( \vclXs,\vclYRt) {};

\node[below left=3pt,xshift=2pt] at (-\vclXs,\vclYLb) {$m_1-z/2$};
\node[above left=1pt] at (-\vclXs,\vclYLt) {$m_1+z/2$};
\node[above left=1pt,yshift=-2pt] at ( \vclXs,\vclYRb) {$m_2-z/2$};
\node[below left=3pt,xshift=3pt] at ( \vclXs,\vclYRt) {$m_2+z/2$};

% Region labels.
\node at (-2,1.10) {$K$};
\node[text=black] at ( 0.12,0.50) {$L^+$};
\node[text=black] at (-0.15,-0.60) {$L^-$};

\end{tikzpicture}
\endgroup
    \caption{Illustration of the planar case.}
    \label{fig:vertical-chord}
\end{figure}

\begin{proof}
Put
\[
    L:=K\div\{-z/2,z/2\}
    =(K-z/2)\cap(K+z/2).
\]

First, $L$ has nonempty interior. Indeed, otherwise the interiors of
$K-z/2$ and $K+z/2$ could be separated by a line. If $n$ is a unit normal
to such a line, then the width of $K$ in the direction $n$ would be at most
$
    |\langle z,n\rangle|
    \leq |z|,
$
contrary to $|z|<w(K)$.

Assume that $z$ is vertical; see Figure \ref{fig:vertical-chord}. Consider the vertical
sections of $K$. Since $L$ has nonempty interior, the set of vertical lines
whose intersection with $K$ contains a line segment of length $|z|$ is a
nondegenerate closed interval in the family of vertical lines.

Let $\ell_1$ and $\ell_2$ be the two boundary lines of this interval. Write
the sections $K\cap\ell_1$ and $K\cap\ell_2$ 
as
\[
    [m_1-z/2,m_1+z/2]
    \qquad\text{and}\qquad
    [m_2-z/2,m_2+z/2].
\]
(If a section is larger than $|z|$, we select its part of length $|z|$.)
All four endpoints of these two line segments belong to $\partial K$,
while their midpoints $m_1$ and $m_2$ belong to $\partial L$.

Let $H$ be the line through $m_1$ and $m_2$. Since $L$ is convex,
$
    [m_1,m_2]\subset L.
$
This line segment divides $L$ into two closed parts, denoted by $L^-$ and
$L^+$. Their outer boundary arcs lie in
$\partial(K+z/2)$ and $\partial(K-z/2),$
respectively.

Let $P_z$ be the open plank bounded by
$H-z/2$ and $H+z/2.$
Its width is at most $|z|$. Moreover, by convexity, up to boundary
line segments,
\[
    K\setminus P_z
    =
    (L^--z/2)
    \mathbin{\sqcup}
    (L^++z/2).
\]
Indeed, the translations by $-z/2$ and $z/2$ send the two outer boundary
arcs of $L$ to $\partial K$ and send $[m_1,m_2]$ to the two boundary
lines of $P_z$.

Since $L^-\cap L^+\subset H$ has zero area,
\[
    \vol_2(K\setminus P_z)
    =
    \vol_2(L^-)+\vol_2(L^+)
    =
    \vol_2(L).
    \qedhere
\]
\end{proof}

\begin{theorem}\label{thm:planar-bodies}
Let $K\subset\E^2$ be a convex body and let $0<W<w(K)$.  Among all finite
families of planks whose total width is at most $W$, the maximum of
\[
    \vol_2\left(K\cap\bigcup_{i=1}^n P_i\right)
\]
is attained by a single plank of width $W$.
\end{theorem}

\begin{proof}
Put
\[
    r:=W/2,
    \qquad
    m_K(W):=
    \min_{u\in S^1}
    \vol_2\bigl(K\div\{-ru,ru\}\bigr).
\]
The minimum exists by compactness of $S^1$ and continuity with respect to
$u$. Choose $u_0\in S^1$ attaining the minimum, and apply
Lemma~\ref{lem:planar-cap-lens} with $z=Wu_0$.  We obtain a plank $P_0$ of
width at most $W$ such that
$
    \vol_2(K\setminus P_0)=m_K(W).
$
If necessary, enlarge $P_0$, keeping the same direction, to a plank $P$ of
width exactly $W$.  Then
$
    \vol_2(K\setminus P)\leq m_K(W).
$
On the other hand, applying Corollary~\ref{cor:uncovered-reduction} to the one-plank
family $\{P\}$ gives
$
    \vol_2(K\setminus P)\geq m_K(W).
$
Thus equality holds.

Now let $P_1,\ldots,P_n$ be any finite family of planks with total width at
most $W$.  Corollary~\ref{cor:uncovered-reduction} gives
\[
    \vol_2\left(K\setminus\bigcup_{i=1}^n P_i\right)
    \geq
    m_K(W)
    =
    \vol_2(K\setminus P).
\]
Equivalently,
\[
    \vol_2\left(K\cap\bigcup_{i=1}^n P_i\right)
    \leq
    \vol_2(K\cap P),
\]
so the single plank $P$ is optimal.
\end{proof}

\section{Discussion}

Corollary~\ref{cor:uncovered-reduction} also allows us to prove
Conjecture~\ref{conj:fixed-width} for a special class of simplices.
Let \(S\) be a nondegenerate simplex such that, for some facet \(F_0\),
the altitude \(h\) corresponding to \(F_0\) equals \(w(S)\).
(For example, among regular simplices in dimensions \(d\geq 2\), only the regular
triangle has this property.)

Applying Lemma \ref{lem:translate-inclusion} to $S$, for $0 < W < h$,
\[
    \vol_d\left(S\div\left\{-\frac{W}{2}u,\frac{W}{2}u\right\}\right)
    \geq
    \left(1-\frac{W}{h}\right)^d\vol_d(S).
\]

Corollary~\ref{cor:uncovered-reduction} consequently implies that every
family of planks \(P_1,\ldots,P_n\) of total width at most \(W\) satisfies
\[
    \vol_d\left(S\setminus\bigcup_{i=1}^n P_i\right)
    \geq
    \left(1-\frac{W}{h}\right)^d\vol_d(S).
\]

On the other hand, a single plank of width \(W\), parallel and adjacent
to \(F_0\), leaves uncovered a simplex homothetic to \(S\) with ratio
\(1-\frac{W}{h}\). 
Thus the bound is attained, and this single plank is optimal.

\section*{}

\subsection*{Funding}
E.B. is supported by the DNRF Chair grant 20231101-28697.
A.P. is partially supported by the NSF grant DMS 2349045.

\subsection*{Declaration on the Use of AI}
The use of the concavity of Minkowski subtraction (within the existing Minkowski subtraction approach by the authors) and the idea for the planar case, emerged through iterative discussions between the authors and ChatGPT by OpenAI.

\bibliographystyle{alpha}
\bibliography{main.bbl}

\end{document}